\documentclass[12pt,reqno]{amsart}

\usepackage{graphicx,subfigure}
\usepackage{hyperref}
\hypersetup{colorlinks=true, citecolor=blue, linkcolor=red}

\numberwithin{equation}{section} \numberwithin{figure}{section}
\numberwithin{table}{section} 
{ \theoremstyle{plain}

}

{ \theoremstyle{definition}
\newtheorem{thm}{Theorem}

\numberwithin{equation}{section} \numberwithin{lem}{section}
\numberwithin{thm}{section} \numberwithin{cor}{section}
\numberwithin{pro}{section} \numberwithin{rem}{section}
}

\begin{document}

\title[A new monotonicity formula]
{A new monotonicity formula for solutions to the elliptic system $\Delta u=\nabla W(u) $}



\author{Christos Sourdis} \address{Department of Mathematics and Applied Mathematics, University of
Crete.}
              \email{csourdis@tem.uoc.gr}           




\maketitle

\begin{abstract}
Using a physically motivated stress energy tensor, we prove weak and
strong monotonicity formulas for solutions  to the semilinear
elliptic system $\Delta u=\nabla W(u)$ with $W$ nonnegative. In
particular, we extend a recent two dimensional result of
\cite{smyrnelis} to all dimensions.
\end{abstract}


Consider the semilinear elliptic system
\begin{equation}\label{eqEq}
\Delta u=\nabla W(u)\ \ \textrm{in}\ \ \mathbb{R}^n,\ \ n\geq 1,
\end{equation}
where $W\in C^3(\mathbb{R}^m; \mathbb{R})$, $m\geq 1$, is
\emph{nonnegative}.

In the scalar case, namely $m=1$, Modica \cite{modica} used the
maximum principle to show that every bounded solution to
(\ref{eqEq}) satisfies the pointwise gradient bound
\begin{equation}\label{eqmodica}
\frac{1}{2}|\nabla u|^2\leq W(u) \ \ \textrm{in}\ \ \mathbb{R}^n,
\end{equation}
(see also \cite{cafamodica}). Using this, together with Pohozaev
identities, it was shown in \cite{modicaProc} that  the following
strong monotonicity property holds:
\begin{equation}\label{eqmonotonicity}
\frac{d}{dR}\left(\frac{1}{R^{n-1}}\int_{B_R}^{}\left\{\frac{1}{2}|\nabla
u|^2+ W\left(u\right) \right\}dx\right)\geq 0,\ \ R>0,
\end{equation}
where $B_R$ stands for the $n$-dimensional ball of radius $R$ that
is centered at  $0$ (keep in mind that (\ref{eqEq}) is translation
invariant).

In the vectorial case, that is  $m\geq 2$, in the absence of the
maximum principle, it is not true in general that the gradient bound
(\ref{eqmodica}) holds (see \cite{smyrnelis} for a counterexample).
Nevertheless, it was shown in \cite{alikakosBasicFacts}, using a
physically motivated stress energy tensor,
 that every solution to
(\ref{eqEq}) satisfies the weak monotonicity property:
\begin{equation}\label{eqmonotoniWeak}
\frac{d}{dR}\left(\frac{1}{R^{n-2}}\int_{B_R}^{}\left\{\frac{1}{2}|\nabla
u|^2+ W\left(u\right) \right\}dx\right)\geq 0,\ \ R>0,\ \ n\geq 2,
\end{equation}
where $|\nabla u|^2=\sum_{i=1}^{n}|u_{x_i}|^2$ (for related results,
obtained via Pohozaev identities, see \cite{bethuelBIG},
\cite{caffareliLin} and \cite{riviere}). In fact, as was observed in
\cite{alikakosBasicFacts}, if $u$ additionally  satisfies the vector
analog of Modica's gradient bound (\ref{eqmodica}), we have the
strong monotonicity property (\ref{eqmonotonicity}).

Interesting applications of these formulas can be found in the
aforementioned references. The  importance of monotonicity formulas
in the study of nonlinear partial differential equations is also
highlighted in the recent article \cite{evansMono}.

Recently, it was proven in \cite{smyrnelis} that, if $u$ is a
bounded solution to the scalar problem with $n=2$, we have that
\begin{equation}\label{eqmonotonicitySMYRN}
\frac{d}{dR}\left(\frac{1}{R}\int_{B_R}^{}W\left(u\right)
dx\right)\geq 0,\ \ R>0.
\end{equation}
This was accomplished by deriving an alternative form of the stress
energy tensor for solutions defined in planar domains, and by giving
a geometric interpretation of  Modica's estimate (\ref{eqmodica}).
We emphasize that the interesting techniques in \cite{smyrnelis} are
intrinsically two dimensional and seem hard to generalize to higher
dimensions.

Interestingly enough, in the vector case, it is stated in
\cite{farinatwores} (without proof) that Pohozaev identities imply
that solutions to the Ginzburg-Landau system
\[
\Delta u=\left(|u|^2-1 \right)u,\ \ u:\mathbb{R}^n\to \mathbb{R}^m,
\ \ \left(\textrm{here}\
W(u)=\frac{\left(1-|u|^2\right)^2}{4}\right),
\]
with $n\geq 2, m\geq 2$, satisfy the weak monotonicity property
\begin{equation}\label{eqmonotoniWeakFarina}
\frac{d}{dR}\left(\frac{1}{R^{n-2}}\int_{B_R}^{}\left\{\frac{n-2}{2}|\nabla
u|^2+ n\frac{\left(1-|u|^2\right)^2}{4} \right\}dx\right)\geq 0,\ \
R>0.
\end{equation}

 It is tempting to wonder whether   there is a strong version of
(\ref{eqmonotoniWeakFarina}), that is with $n-1$ in place of $n-2$,
in the scalar case (for any smooth $W\geq 0$), which for $n=2$ gives
(\ref{eqmonotonicitySMYRN}). In this note, by appropriately
modifying the systematic approach of \cite{alikakosBasicFacts}, we
prove the following general result which, in particular, confirms
this connection.

\begin{thm}
If $u\in C^2(\mathbb{R}^n;\mathbb{R}^m)$, $n\geq 2, m\geq 1$, solves
(\ref{eqEq}) with $W\in C^1(\mathbb{R}^m;\mathbb{R})$
\emph{nonnegative}, we have the weak monotonicity formula:
\begin{equation}\label{eqmonotoniMYWEAK}
\frac{d}{dR}\left(\frac{1}{R^{n-2}}\int_{B_R}^{}\left\{\frac{n-2}{2}|\nabla
u|^2+ nW(u) \right\}dx\right)\geq 0,\ \ R>0.
\end{equation}
In addition, if $u$ satisfies Modica's gradient bound, that is
\begin{equation}\label{eqmodicaVector}
\frac{1}{2}|\nabla u|^2\leq W(u) \ \ \textrm{in}\ \ \mathbb{R}^n,
\end{equation}
we have the strong monotonicity formula:
\begin{equation}\label{eqmonotoniMYWEAK}
\frac{d}{dR}\left(\frac{1}{R^{n-1}}\int_{B_R}^{}\left\{\frac{n-2}{2}|\nabla
u|^2+ nW(u) \right\}dx\right)\geq 0,\ \ R>0.
\end{equation}
\end{thm}
\begin{proof}
By means of a direct calculation, it was shown in
\cite{alikakosBasicFacts} that, for solutions $u$ to (\ref{eqEq}),
the stress energy tensor $T(u)$, which is defined as the $n\times n$
matrix with entries
\[
T_{ij}=u_{,i}\cdot u_{,j}- \delta_{ij}\left(\frac{1}{2}|\nabla u|^2+
W\left(u\right) \right),\ \ i,j=1,\cdots,n,\ (\textrm{where}\
u_{,i}=u_{x_i}),\]
 satisfies
\begin{equation}\label{eqdiv}
\textrm{div}T(u)=0,
\end{equation}
using the notation $T=(T_1,T_2,\cdots,T_n)^\top$ and
$\textrm{div}T=(\textrm{div}
T_1,\textrm{div}T_2,\cdots,\textrm{div}T_n)^\top$, (see also
\cite{serfaty}). Observe that
\begin{equation}\label{eqtrace}
\textrm{tr}T=-\left(\frac{n-2}{2}|\nabla u|^2+ nW(u) \right),
\end{equation}
and that
\begin{equation}\label{eqpositiv}
T+\left(\frac{1}{2}|\nabla u|^2+ W\left(u\right) \right) I_n=(\nabla
u)^\top(\nabla u)\geq 0 \ \ \ \textrm{(in the matrix sense)},
\end{equation}
where $I_n$ stands for the $n\times n$ identity matrix.

As in \cite{schoen}, writing $x=(x_1,\cdots,x_n)$, and making use of
(\ref{eqdiv}), we calculate that
\begin{equation}\label{eq1}
\sum_{i,j=1}^{n}\int_{B_R}^{}\left(x_i
T_{ij}\right)_{,j}dx=\sum_{i,j=1}^{n}\int_{B_R}^{}\left\{\delta_{ij}T_{ij}+
x_{i}(T_{ij})_{,j}\right\}dx=\sum_{i=1}^{n}\int_{B_R}^{} T_{ii}dx.
\end{equation}
On the other side, from the divergence theorem, denoting $\nu=x/R$,
and making use of (\ref{eqpositiv}), we find that
\begin{equation}\label{eq2}
\sum_{i,j=1}^{n}\int_{B_R}^{}\left(x_i
T_{ij}\right)_{,j}dx=R\sum_{i,j=1}^{n}\int_{\partial B_R}^{}\nu_i
T_{ij} \nu_j dS\geq -R\int_{\partial B_R}^{}\left(\frac{1}{2}|\nabla
u|^2+ W\left(u\right) \right)dS.
\end{equation}
Since $W$ is nonnegative, if $n\geq 3$, we have that
\begin{equation}\label{eqineq1}
\frac{1}{2}|\nabla u|^2+ W\left(u\right)\leq \frac{1}{n-2}
\left(\frac{n-2}{2}|\nabla u|^2+ n W\left(u\right) \right).
\end{equation}

Let
\[
f(R)=\int_{B_R}^{}\left(\frac{n-2}{2}|\nabla u|^2+ n W\left(u\right)
\right) dx,\ \ R>0.
\]
By combining (\ref{eqtrace}), (\ref{eq1}), (\ref{eq2}) and
(\ref{eqineq1}), for $n\geq 3$, we arrive at
\[
-f (R)\geq -\frac{R}{n-2}\frac{d}{dR}f(R),\ \ R>0,
\]
which implies that
\[
\frac{d}{dR}\left(R^{2-n}f(R) \right)\geq 0,\ \ R>0,
\]
(clearly this also holds for $n=2$). We have thus shown the first
assertion of the theorem.

Suppose that $u$ additionally satisfies Modica's gradient bound
(\ref{eqmodicaVector}). Then, we can strengthen (\ref{eqineq1}), for
$n\geq 2$, by noting that
\[
\frac{1}{2}|\nabla u|^2+ W\left(u\right)= \frac{1}{n-1}
\left(\frac{n-2}{2}|\nabla u|^2+\frac{1}{2}|\nabla u|^2+ (n-1)
W\left(u\right) \right)\leq \frac{1}{n-1} \left(\frac{n-2}{2}|\nabla
u|^2+ n W\left(u\right) \right).
\]
Now, by combining (\ref{eqtrace}), (\ref{eq1}), (\ref{eq2}) and the
above relation, we arrive at
 \[ -f (R)\geq -\frac{R}{n-1}\frac{d}{dR}f(R),\ \ R>0,
\]
which implies that
\[
\frac{d}{dR}\left(R^{1-n}f(R) \right)\geq 0,\ \ R>0,
\]
as desired.
\end{proof}

\end{document}